\documentclass{amsart}

\usepackage[UKenglish]{babel}

\newcommand{\simpleoverline}[1]{\dot{#1}}
\newcommand{\doubleoverline}[1]{\ddot{#1}}
\newcommand{\tripleoverline}[1]{\dddot{#1}}
\newcommand{\quadrupleoverline}[1]{\ddddot{#1}}

\newcommand{\Quad}{\operatorname{Quad}}
\newcommand{\Cub}{\operatorname{Cub}}
\newcommand{\Quart}{\operatorname{Quart}}

\renewcommand{\th}{^{\rm th}}

\newcommand{\inmargin}[1]{}

\theoremstyle{plain}
\newtheorem*{notation*}{Notation}
\newtheorem*{relations*}{Relations}
\newtheorem*{question*}{Question}
\newtheorem*{observation*}{Observation}
\newtheorem{lemma}{Criterion}
\newtheorem{theorem}{Theorem}
\newtheorem*{corollary*}{Corollary}
\newtheorem{proposition}{Proposition}
\newtheorem*{fact*}{Fact}
\newtheorem*{TQT*}{Timmesfeld's Quadratic Theorem}
\newtheorem*{theoremone*}{Theorem {\ref{t:SL2:series}}}
\newtheorem*{theoremtwo*}{Theorem {\ref{t:Symn-1NatSL2K}}}
\theoremstyle{remark}
\newtheorem{step}{Claim}[subsubsection]
\newtheorem*{localnotation*}{Notation}
\newtheorem*{localobservation*}{Observation}
\newtheorem*{localrelations*}{Relations}
\newtheorem*{localremark*}{Remark}
\newtheorem*{localremarks*}{Remarks}
\newtheorem{localconsequence}{Consequence}
\theoremstyle{definition}
\newtheorem*{remark*}{Remark}
\newtheorem*{remarks*}{Remarks}

\makeatletter
\newenvironment{proofclaim}
{\par\pushQED{\qed}
\normalfont \topsep6\p@\@plus6\p@\relax\trivlist
\item[\hskip\labelsep
\emph{Proof of Claim.}]
\ignorespaces

}
{\popQED\endtrivlist\@endpefalse}
\makeatother

\newcommand{\<}{\langle}
\renewcommand{\>}{\rangle}

\newcommand{\Z}{\mathbb{Z}}
\newcommand{\Q}{\mathbb{Q}}
\newcommand{\F}{\mathbb{F}}
\newcommand{\K}{\mathbb{K}}

\newcommand{\fg}{\mathfrak{g}}

\newcommand{\im}{\operatorname{im}}
\newcommand{\Id}{\operatorname{Id}}

\newcommand{\SL}{\operatorname{SL}}
\newcommand{\PSL}{\operatorname{PSL}}
\renewcommand{\sl}{\mathfrak{sl}}
\newcommand{\Ann}{\operatorname{Ann}}
\newcommand{\End}{\operatorname{End}}
\newcommand{\Sym}{\operatorname{Sym}}
\newcommand{\Nat}{\operatorname{Nat}}

\title{Symmetric Powers of Nat SL(2,K)}
\author{Adrien Deloro}

\begin{document}

\maketitle

\makeatletter
\def\@fnsymbol#1{\ensuremath{\ifcase#1\or \or *\or \dagger\or \ddagger\or
   \mathsection\or \mathparagraph\or \|\or **\or \dagger\dagger
   \or \ddagger\ddagger \else\@ctrerr\fi}}
\makeatother

\renewcommand{\thefootnote}{\fnsymbol{footnote}}

\renewcommand{\>}{\rangle}
\renewcommand{\sl}{\mathfrak{sl}}
\renewcommand{\th}{^{\rm th}}

\begin{flushright}
\em In magnis et voluisse sat est.
\end{flushright}
\bigskip

\begin{quote}
\textbf{Abstract.} We identify the representations $\K[X^k, X^{k-1}Y, \dots, Y^k]$ among abstract $\Z[\SL_2(\K)]$-modules. One result is on $\Q[\SL_2(\Z)]$-modules of short nilpotence length and generalises a classical ``quadratic'' theorem by Smith and Timmesfeld. Another one is on extending the linear structure on the module from the prime field to $\K$.
All proofs are by computation in the group ring using the Steinberg relations.
\end{quote}

\footnote{

Institut de Math\'ematiques de Jussieu -- Paris Rive Gauche

adrien.deloro@imj-prg.fr}


We study here certain representations of the group $\SL_2(\K)$ as an abstract group; more precisely, we aim at identifying the various symmetric powers of $\Nat \SL_2(\K)$, conveniently thought of as the various spaces of homogeneous polynomials in two variables with fixed degree, among $\Z[\SL_2(\K)]$-modules.
Differently put, we study the inclusion of the class of representations of the algebraic group $\SL_2$ over the field $\K$, in the wider class of $\Z[\SL_2(\K)]$-modules. The question may sound not quite irrelevant to admirers of the Borel-Tits Theorem on abstract homomorphisms between groups of points of algebraic groups; we deal with abstract modules instead.

We cannot use Lie-theoretic, algebraic geometric, nor character-theoretic methods since $\SL_2(\K)$ is to us but an abstract group and $\K$ is arbitrary. We cannot even use linear algebra since we do not assume our modules to be vector spaces.
Our only method is then brute force computation in images of the group ring.
So the problem rephrases into: To which extent is the representation theory of $\SL_2(\K)$ determined by the ``inner'' group-theoretic constraints?

The present study is therefore yet another instance of the general problem of investigating representations of algebraic groups from a purely group-theoretic perspective, which we tackled in \cite{TV-I} and \cite{TV-II}. It can however be read independently of the latter two articles and was written in this intention.

One should simply recall a result first proved by F. G. Timmesfeld and S. Smith separately.
In what follows, $\Nat$ stands for the natural representation, here the action of $\SL_2(\K) = \SL(\K^2)$ on $\K^2$. Moreover $U$ stands for a unipotent subgroup of $\SL_2(\K)$, and the assumption on the $U$-length being $2$ means that $U$ acts \emph{quadratically}: for all $u_1, u_2 \in U$, one has $(u_1 -1)(u_2 - 1) = 0$ in $\End(V)$.
One word on this assumption -- since we are dealing with abstract modules instead of vector spaces, there is no dimension around. Unipotence length is then the natural candidate to measure the complexity of target modules; the length of $\Nat \SL_2(\K)$ is $2$ (and more generally the length of $\Sym^k \Nat \SL_2(\K)$ is also its dimension over $\K$, namely $k+1$).

\begin{TQT*}[{\cite[Theorem 3.4 of chapter I]{Timmesfeld}, also \cite{Smith}}]
Let $\K$ be a field, $G = \SL_2(\K)$, and $V$ be a simple $\Z[G]$-module of $U$-length $2$. Then there exists a $\K$-vector space structure on $V$ making it isomorphic to $\Nat \SL_2(\K)$.
\end{TQT*}

Our original motivation was to find a similar result identifying the adjoint representation, viz. the action on $2\times 2$ matrices with null trace, among cubic $G$-modules, i.e. modules of $U$-length exactly $3$, with an obvious definition.
As a matter of fact, our very first step towards the adjoint representation was a joint work with G. Cherlin in the context of model theory \cite{CDSmall}.
The present work could be taken as an insane expansion of the former; see our final corollary.
(Very parenthetically said, but still in length $3$: our work is independent from the more general study led by M. Gr\"uninger \cite{GCubic}, which takes place in Timmesfeld's theory of abstract ``rank one groups'' \cite{Timmesfeld}. Gr\"uninger deals with abstract groups not necessarily isomorphic with $\SL_2(\K)$, and this loose assumption leads to numerous difficulties we ignore by being restrictive on the group. These parentheses were then nothing but a digression since we shall focus on $\SL_2(\K)$.)

Returning to representations of $\SL_2(\K)$ seen as abstract modules, we divided the problem into two tasks: first deal with the prime field $\K_1$, then go up from the prime field $\K_1$ to the extension field $\K$.
The dominant inspiration for doing so was the idea arguably due to Chevalley that the group $\SL_2(\K)$ is a vertebrate animal, viz. with an endoskeleton and then flesh on it: that is, that fundamental group-theoretic constraints can be seen at the level of the subgroup of points over the prime subfield, and that these bony relations are naturally clad in well-rounded copies of the field.
Our two main results stated below reflect this two-step methodology; notice that the first tried to be excessively ambitious and pretended to analyse the skeleton at the level of the very integers.

\begin{theoremone*}
Let $V$ be a $\Q[\SL_2(\Z)]$-module. Suppose that for every unipotent element $u \in \SL_2(\Z)$, $(u-1)^5 = 0$ in $\End V$. Then $V$ has a composition series each factor of which is a direct sum of copies of $\Q\otimes_\Z \Sym^k \Nat \SL_2(\Z)$ for $k \in \{0, \dots, 4\}$.
\end{theoremone*}

Theorem \ref{t:SL2:series} is proved in \S\ref{S:SL2:bones} by an excessively painful computation which Maxime Wolff could legitimate, but not eliminate, with a geometric argument reproduced in \S\ref{s:SL2:geometry}. This leaves us with a number of questions:
\begin{itemize}
\item
What happens to Theorem \ref{t:SL2:series} when one takes $n = 6$ instead of $5$?
\item
What happens to Theorem \ref{t:SL2:series} with $\Q$ instead of $\Z$ and no bound on $n$?
\item
Does Wolff's geometric argument contain, or suggest, a less computational proof of Theorem \ref{t:SL2:series}?
\item
Does the computation in \S\ref{S:SL2:bones} contain, or bear, some geometry (in any sense)?
\end{itemize}
Let us be honest: the computation is a complete mystery to us and we wish to ask the community what its meaning can be.
The paper is a call for help and we will be delighted to offer a bottle of Scotch whisky to anyone explaining what is going on.
As for the behaviour over $\F_p$ instead of $\Q$ (with no bounds on $n$), we do not know either but this should be classical.

On the second task, namely predicting the structure of an $\SL_2(\K)$-module just by looking at the restricted $\SL_2(\K_1)$-module structure where $\K_1$ is the prime subfield, we obtained the following. The double factorial is defined by $n!! = (n-2)!!$ and $\oplus_I M$ means a direct sum of copies of $M$ indexed by some set $I$.

\begin{theoremtwo*}
Let $n\geq 2$ be an integer and $\K$ be a field of characteristic $0$ or $\geq 2n+1$. Suppose that $\K$ is $2(n-1)!!$-radically closed. Let $G = \SL_2(\K)$ and $V$ be a $G$-module. Let $\K_1$ be the prime subfield and $G_1 = \SL_2(\K_1)$. Suppose that $V$ is a $\K_1$-vector space such that $V \simeq \oplus_I \Sym^{n-1} \Nat G_1$ as $\K_1 [G_1]$-modules.

Then $V$ bears a compatible $\K$-vector space structure for which one has $V \simeq\oplus_J \Sym^{n-1}\Nat G$ as $\K [G]$-modules.
\end{theoremtwo*}

Theorem \ref{t:Symn-1NatSL2K} is proved in \S\ref{S:SL2:flesh} by a lighter computation which goes so smoothly that there may be something more general to look for.

Parenthetically said, Theorems \ref{t:SL2:series} and \ref{t:Symn-1NatSL2K} may be compared with the conclusions of \cite{TV-II}, a study of $\Z[\sl_2(\K)]$-modules where $\sl_2(\K)$ is the set of $2\times 2$ matrices with null trace seen as a Lie \emph{ring}, i.e. endowed with an addition and a Lie bracket but no vector space structure. We followed the two-step methodology discussed above; as one shall see the skeleton of the Lie ring is much more rigid than that of the group, arguably because of the Casimir element.


\begin{fact*}[{\cite[Variations n$^\circ$ 17 and n$^\circ$18]{TV-II}.}]
Let $n \geq 2$ be an integer and $\K_1$ be a \emph{prime} field of characteristic $0$ or $\geq n+1$.
Let $\fg_1 = \sl_2(\K_1)$ and $V$ be a $\fg_1$-module. If the characteristic of $\K$ is $0$ one requires $V$ to be torsion-free. Suppose that $x^n = 0$ in $\End V$; if $\K_1$ has characteristic $p$ with $n < p < 2n$, suppose further that $y^n = x^n = 0$ in $\End V$.

Then $V = \Ann_V(\fg_1) \oplus \fg_1 \cdot V$, and $\fg_1\cdot V$ is a $\K_1$-vector space with $\fg_1 \cdot V \simeq \oplus_{k = 1}^{n-1} \oplus_{I_k} \Sym^k \Nat \fg_1$ as $\K_1\fg_1$-modules.
\end{fact*}

\begin{fact*}[{\cite[Variation n$^\circ$19]{TV-II}.}]
Let $n \geq 2$ be an integer and $\K$ be a field of characteristic $0$ or $\geq n$.
Let $\fg = \sl_2(\K)$ viewed as a Lie ring and $V$ be a $\fg$-module. Let $\K_1$ be the prime subfield of $\K$ and $\fg_1 = \sl_2(\K_1)$. Suppose that $V$ is a $\K_1$-vector space such that $V \simeq \oplus_I \Sym^{n-1} \Nat \fg_1$ as $\K_1\fg_1$-modules.

Then $V$ bears a compatible $\K$-vector space structure for which one has $V \simeq \oplus_J \Sym^{n-1} \Nat \fg$ as $\K\fg$-modules.
\end{fact*}

These results are just mentioned and will not be used. Before we start 
we wish to thank: Antonin Guilloux and Maxime Wolff (see \S\ref{s:SL2:geometry}) on the one hand for their geometric help, and Alexandre Borovik and Gregory Cherlin on the other hand, who patiently endured earlier and even longer computations.

\section{Combinatorial Skeleton}\label{S:SL2:bones}

In this section we study $\SL_2(\Z)$-modules of short length. The main result is Theorem \ref{t:SL2:series} from the introduction, which we prove by a most brutal computation in \S\ref{s:SL2:computation}.
Allow us to insist that for us $\SL_2(\Z)$ is nothing but a pure group; 
we do not endow it with structure inherited from the algebraic group functor $\SL_2$, and must therefore do clumsy, ``pedestrian'' identification.

\begin{notation*}
Let $G_0 = \SL_2(\Z)$.\inmargin{$G_0$}
\end{notation*}

We let $\Nat \SL_2(\Z)$ stand for $\Z^2$ as the natural $\Z[\SL_2(\Z)]$-module, and we also let $\Sym^k \Nat \SL_2(\Z)$ stand for its $k\th$ symmetric power. Such modules do not have good divisibility properties, so we shall be interested in the tensored $\Q[\SL_2(\Z)]$-modules $\Q\otimes_\Z \Sym^k \Nat \SL_2(\Z)$.

\begin{notation*}
Let \inmargin{$\Sym_\Q^k$} $\Sym_\Q^k \Nat G_0 = \Q\otimes_\Z \Sym^k \Nat \SL_2(\Z)$.
\end{notation*}

Hence $\Sym_\Q^k \Nat G_0$ is the $(k+1)$-dimensional space spanned by $X^k$, $X^{k-1}Y$, \dots, $Y^k$ over $\Q$ and endowed with the usual action of $\SL_2(\Z) \leq \SL_2(\Q)$ on polynomials.

\S\ref{s:SL2:criterion} yields a trivial criterion used in the highly computational \S\ref{s:SL2:computation}. \S\ref{s:SL2:geometry} is a meditation on the geometric contents of the latter, a meditation entirely due to Maxime Wolff. And since we reach a dead-end, further questions we mentioned in the introduction are suggested in \S\ref{s:fenetre}.

\subsection{Notations and Criteria}\label{s:SL2:criterion}

Criterion \ref{l:SL2:criterion2} below will be used systematically in \S\ref{s:SL2:computation} to prove Theorem \ref{t:SL2:series}. We need a few notations.

\begin{notation*}
Let $u = \begin{pmatrix}1 & 1 \\ 0 & 1\end{pmatrix}$ \inmargin{$u$}
 and $w = \begin{pmatrix}0 & 1 \\ -1 & 0 \end{pmatrix}$.\inmargin{$w$}
\end{notation*}

We know that $i = w^2$ generates $Z(G_0)$.\inmargin{$i$}

\begin{relations*}[Steinberg relations]
$(uw)^3 = 1$.
\end{relations*}

The length $\ell(V)$\inmargin{$\ell(V)$} of a $G_0$-module $V$ is the least (if any) $k$ with $(u-1)^k \cdot V = 0$.

\begin{notation*}
If $V$ is $\ell(V)!$-divisible and $\ell(V)!$-torsion-free, let $x = \log u \in \End V$.
\end{notation*}

\begin{lemma}\label{l:SL2:criterion1}
Let $n\geq 2$ be an integer and $V$ be a $\Q[\SL_2(\Z)]$-module of length $\leq n$. Suppose that for all $k = 1\dots n$ one has in $\End V$:
\[\frac{1}{(n-k)!}wx^{n-k}wx^{n-1} = \frac{(-1)^{n-k}}{(k-1)!} x^{k-1}wx^{n-1}\]

Then $V$ has a $\Q[\SL_2(\Z)]$-submodule $V_\top$ such that $V/V_\top$ has length $\leq n-1$, and $V_\top \simeq \bigoplus_I \Sym_\Q^{n-1} \Nat \SL_2(\Z)$ as $\Q[\SL_2(\Z)]$-modules.
\end{lemma}
\begin{proof}[Proof sketch]
For $k = 1 \dots n$, the maps $\pi_k = \frac{1}{((n-1)!)^2} x^{n-k} w x^{n-1} w x^{k-1}$ are orthogonal idempotents; $V_\top = \oplus \im \pi_k$ is a $\Q[G_0]$-submodule, and $V_\top = \<G_0\cdot \im \pi_1\>_\Q$.
\end{proof}

\begin{remark*}
Here is a dual statement: if $V$ has length $\leq n$ and in $\End V$ holds $\frac{1}{(n-k)!}x^{n-1}wx^{n-k}w = \frac{(-1)^{n-k}}{(k-1)!} x^{n-1}wx^{k-1}$, then $V$ has a submodule $V_\bot$ of length $\leq n$ such that the quotient $V/V_\bot$ is isomorphic, etc.

Note that under the assumptions of Criterion \ref{l:SL2:criterion1} one can define the subgroup $V_\bot$ as $\cap_{k = 1}^n \ker \pi_k$, and that one does have $\im \left( 1 - \sum_{k = 1}^n \pi_k\right) \leq V_\bot$. But it is not clear whether $V_\bot$ is $G_0$-invariant. Our ``dual'' assumption forces this as a simple computation shows.

One could also argue by duality. In general, if $V$ is a $\Q[\SL_2(\Z)]$-module of finite length then so is the dual space $V^\ast$, and the following holds. 
Let $b$ be a word in $x$ and $w$ and $d$ be the word written in reverse order; let $(v, f) \in V \times V^\ast$. Then $(b\cdot f)(v) = (-1)^r \cdot f(i^s d\cdot v)$ where $i$ is the central involution, and the integers $r$ and $s$ are easily computed from $b$.

Here, one can check that if $V$ satisfies the dual assumption $\frac{1}{(n-k)!}x^{n-1}wx^{n-k}w = \frac{(-1)^{n-k}}{(k-1)!} x^{n-1}wx^{k-1}$, then the dual module $V^\ast$ satisfies the assumptions of Criterion \ref{l:SL2:criterion1}. Hence $V^\ast$ has a submodule $V_\top^\ast$ with the desired properties. One then sets $V_\bot = (V_\top^\ast)^\perp = \{v \in V: \forall \varphi \in V_\top^\ast, \varphi(v) = 0\}$, which meets the requirements.
\end{remark*}

\begin{lemma}\label{l:SL2:criterion2}
Let $n\geq 2$ be an integer and $V$ be a $\Q[\SL_2(\Z)]$-module of length $\leq n$. Suppose that for all $k = 1\dots n$ one has in $\End V$:
\[\frac{1}{(n-k)!}wx^{n-k}wx^{n-1} = \frac{(-1)^{n-k}}{(k-1)!} x^{k-1}wx^{n-1}\]
Suppose further either $\ker x \cap \ker(x^{n-1}w) = 0$, or $V = \im x + \im(wx^{n-1})$.

Then $V \simeq \bigoplus_I \Sym_\Q^{n-1} \Nat \SL_2(\Z)$ as $\Q[\SL_2(\Z)]$-modules.
\end{lemma}

\begin{proof}[Proof sketch]
%
%
In the notations of Criterion \ref{l:SL2:criterion1}, it suffices to see $V = V_\top$. This is clear if $V = \im x + \im(wx^{n-1})$; if $\ker x\cap \ker(x^{n-1}w) = 0$, let $q_k = ((n-1)!)^2 x^k \left( \sum_{\ell = 1}^n \pi_\ell - 1\right)$ and prove that for $k = n \dots 0$, $q_k = 0$, so $\sum \pi_k = 1$.
\end{proof}

\begin{remark*}
Since the equation in the assumption is not self-dual, one of the two arguments would not suffice to prove Criterion \ref{l:SL2:criterion2}.

Here is a dual statement: if in $\End V$ one has $\ker x \cap \ker(x^{n-1}w) = 0$ or $V = \im x + \im(wx^{n-1})$, and $\frac{1}{(n-k)!}x^{n-1}wx^{n-k}w = \frac{(-1)^{n-k}}{(k-1)!} x^{n-1}wx^{k-1}$, then we reach the same conclusion as in Criterion \ref{l:SL2:criterion2}.

This is because if $V$ is a $\Q[\SL_2(\Z)]$-module of finite length, then (setting $Z_k(W) = \ker (x^k)$ when acting on $W$): $V = \im x + \im(wx^{n-1})$ iff $Z_1(V^\ast)\cap w\cdot Z_{n-1}(V^\ast) = 0$, and $Z_1(V) \cap w \cdot Z_{n-1}(V) = 0$ iff $V^\ast = \im x + \im(wx^{n-1})$.
\end{remark*}

Of course there are similar statements for $\F_p[\SL_2(\F_p)]$-modules if $p > n$.

\subsection{The Long Computation}\label{s:SL2:computation}

The present \S\ref{s:SL2:computation} is dedicated to proving Theorem \ref{t:SL2:series} by means of a tedious computation and is compteley void of ideas. A reader not enjoying heavy calculations should skip it and jump to \S\ref{s:SL2:geometry}.

\begin{theorem}\label{t:SL2:series}
Let $V$ be a $\Q[\SL_2(\Z)]$-module. Suppose that for every unipotent element $u \in \SL_2(\Z)$, $(u-1)^5 = 0$ in $\End V$. Then $V$ has a composition series each factor of which is a direct sum of copies of $\Q\otimes_\Z \Sym^k \Nat \SL_2(\Z)$ for $k\in \{0, \dots, 4\}$.
\end{theorem}

\begin{remarks*}\
\begin{itemize}
\item
If $(u-1)^3 = 0$ the series even splits: $V$ is a direct sum of submodules of the desired type. We shall check it in due time.
%
%
\item
Powers $k$ in Theorem \ref{t:SL2:series} may appear with repetitions. 
We do not even know whether terms can be rearranged in non-decreasing power order.
\item
We shall not use all of the $\Q$-vector space structure during our computations. A $\Z_{\frac{1}{n!}}$-module is enough to derive our formulas. In particular, Theorem \ref{t:SL2:series} has an analogue for $\F_p[\SL_2(\F_p)]$-modules ($p > 5$) -- which we suspect could also be obtained with much less effort.
\end{itemize}
\end{remarks*}

\begin{proof}[The proof of Theorem \ref{t:SL2:series} starts here]
Writing $V = C_V(i) \oplus [V, i]$, we may assume that $i = \pm 1$ in $\End(V)$.
We shall build the series inductively. For $V$ of length $\ell$ we construct a non-maximal series of submodules $0 = V_0 < \dots < V_m = V$ such that:
\begin{itemize}
\item
for $j < m$, $V_j/V_{j-1}$ has length $< \ell$,
\item
$V/V_{m-1}$ either has length $< \ell$, or satisfies the assumptions of Criterion \ref{l:SL2:criterion2} (depending on the value of the involution in $\End V$).
\end{itemize}

\subsubsection{Notations and Remarks}

In order to analyse modules we need to isolate a ``quadratic'' radical, a ``cubic'' radical, and so on. This requires a few notations.

\begin{localnotation*}
Let \inmargin{$Z_j(V)$, $Z_j^k(V)$} $Z_j(V) = \ker (u-1)^j$ and $Z_j^k(V) = Z_j(V) \cap w\cdot Z_k(V)$.
\end{localnotation*}

For instance $Z_1^1(V) = \ker(u-1) \cap \ker((u-1)w) = C_V(u, wuw^{-1}) = C_V(G_0)$.
We have let $x = \log (1+(u-1))$, so that $u = \exp(x)$.
Clearly
 $Z_k (V) = \ker (x^k)$.

\begin{localnotation*}
Let\inmargin{$\Quad, \Cub$}
$\Quad(V) = Z_1^2(V) + Z_2^1(V)$ and 
$\Cub(V) = Z_1^3(V) + Z_2^2(V) + Z_3^1(V)$.
\end{localnotation*}

Nothing guarantees that the $\Q$-vector subspaces $\Quart(V)$ and $\Cub(V)$ are $\Q[G_0]$-submodules: we prove it as follows.

\begin{localnotation*}
Let \inmargin{$c$, $s$} $c = \cosh(x)$ and $s = \sinh(x)$, so that $u = c + s$ and $u^{-1} = c - s$.
\end{localnotation*}

\begin{localrelations*}
If $i = 1$ in $\End V$, then
$wcw  = cwc + sws$ and $wsw = - cws - swc$.

If on the other hand $i = -1$, then
$wcw = cws + swc$ and $wsw = - cwc - sws$.
\end{localrelations*}
\begin{proofclaim}
In $\End V$ one has by the Steinberg relations $uwu = (wuw)^{-1} = wu^{-1} w$ and $u^{-1}wu^{-1} = i(uwu)^{-1} = i wu w$.
\end{proofclaim}

\begin{localobservation*}
$\Quad(V)$ is always $G_0$-invariant; if $i = 1$ then so is $\Cub(V)$.
\end{localobservation*}
\begin{proofclaim}
First suppose $i = -1$. Let us show that $\Quad(V)$ is $G_0$-invariant. Its $w$-invariance is obvious (and will no longer be mentioned in similar arguments). Clearly $x$ maps $Z_1^2(V)$ to $\Quad(V)$. Finally if $a_2 = w b_1 \in Z_2^1(V)$, then:
\begin{align*}
x^2 wx a_2 & = - x^2 wsww a_2 = x^2 cwc wa_2 + x^2 sws w a_2 = x^2 c ww a_2 = 0
\end{align*}
so $x a_2 \in Z_1^2(V) \leq \Quad(V)$, and this shows that $x$ maps $\Quad(V)$ to itself: the latter is therefore $\<u, w\> = G_0$-invariant.

We now suppose $i = 1$. To prove $G_0$-invariance of $\Quad(V)$ we argue similarly and take $a_2 \in Z_2^1(V)$:
\[x^2 wx a_2 = x^2 w sww a_2 = - x^2 cws wa_2 - x^2 swc wa_2 = - x^2 s ww a_2 = 0\]
which shows $G_0$-invariance of $\Quad(V)$.

To prove $G_0$-invariance of $\Cub(V)$ (still assuming $i = 1$ in $\End V$) there are two non-trivial verifications. First let $a_3 \in Z_3^1(V)$. Then:
\[x^2 w x a_3 = x^2 ws ww a_3 = - x^2 cwswa_3 - x^2 swc wa_3 = - x^2 s a_3 = 0\]
so $x a_3 \in Z_2^2(V) \leq \Cub(V)$. Now let $a_2 \in Z_2^2(V)$. Decomposing under the action of the involution $w$, we may assume that $w a_2 = \pm a_2$, say $w a_2 = \varepsilon a_2$. Hence:
\[xwx a_2 = \varepsilon xws w a_2 = - \varepsilon xcws a_2 - \varepsilon xswc a_2 = - \varepsilon xcwxa_2 - \varepsilon xsw a_2 = - \varepsilon xcwxa_2\]
So $(1 + \varepsilon c) x wx a_2 = 0$. In any case $x^3 wxa_2 = 0$, so $xa_2 \in Z_1^3(V)$ as desired.
\end{proofclaim}

\subsubsection{General Formula}


\begin{localrelations*}
If $i = 1$ in $\End(V)$, then:
\begin{align*}\label{e:+}
0 & = - 3 s - 3ws - 3sw + 3cws + 3 swc + \frac{1}{2} x^3 wu + \frac{1}{2} u^{-1}wx^3 +  \frac{1}{2} uwx^3w\\
& \quad  + \frac{1}{2} wx^3wu^{-1} + \frac{1}{4} x^4wu - \frac{1}{4}u^{-1}wx^4 + \frac{1}{4} uwx^4w - \frac{1}{4} wx^4wu^{-1}\tag{$E_+$}
\end{align*}
If on the other hand $i = -1$, then:
\begin{align*}\label{e:-}
0 & = 3cws + 3swc +3cw - 3wc + 3c + \frac{1}{2}x^3wu + \frac{1}{2}uwx^3w + \frac{1}{2}u^{-1}wx^3 \\
& \quad - \frac{1}{2}wx^3wu^{-1} + \frac{1}{4} x^4wu + \frac{1}{4}wx^4wu^{-1} - \frac{1}{4} u^{-1}wx^4 + \frac{1}{4}uwx^4w\tag{$E_-$}
\end{align*}
\end{localrelations*}
\begin{proofclaim}
Since the length is at most $5$, one sees that:
\[
c^2 = 2 c - 1 + \frac{1}{4}x^4; \qquad 
s^2 = 2 c - 2 + \frac{1}{4} x^4; \qquad
cs = s + \frac{1}{2}x^3
\]

First suppose $i = 1$ and get ready for a long computation.
\begin{align*}
0 & = \left(uwu-wu^{-1}w\right)cw\\
& = uw \left(c^2 + cs\right) w - wu^{-1} cwc - wu^{-1} sws\\
& = uwc^2 w + uwcsw - wc^2wc + wcswc - wcsws + ws^2 ws\\
& = u + uws^2 w + uwcsw + wcs wu^{-1} - c - ws^2 wc + ws^2 ws\\
& = s + uws^2 w - ws^2 wu^{-1} + uwcsw + wcswu^{-1}\\
& = s + uw \left(2c-2+\frac{1}{4}x^4\right) w - w\left(2c-2 + \frac{1}{4}x^4\right)wu^{-1} + uw\left(s+\frac{1}{2}x^3\right)w\\
& \quad + w\left(s+\frac{1}{2}x^3\right) wu^{-1}\\
& = s + 2ucwc + 2usws - 2u + \frac{1}{4} uwx^4w - 2cwcu^{-1} - 2 swsu^{-1} + 2u^{-1}\\
& \quad  - \frac{1}{4} wx^4 wu^{-1} - ucws - uswc + \frac{1}{2} uwx^3w - cwsu^{-1} - swcu^{-1} + \frac{1}{2} wx^3 wu^{-1}
\end{align*}
Set $R = \frac{1}{2} uwx^3w + \frac{1}{2} wx^3wu^{-1} + \frac{1}{4} uwx^4 w - \frac{1}{4}wx^4 wu^{-1}$ and resume.
\begin{align*}
0 & = s + 2 wc + 2s^2wc + 2cswc + 2csws + 2s^2ws - 2u - 2cw - 2cws^2 + 2cwcs\\
& \quad  - 2 swcs + 2sws^2 + 2u^{-1} - ws - s^2ws - csws - cswc - s^2 wc - cwcs\\
& \quad + cws^2 - sw - sws^2 + swcs + R\\
& = -3s + 2wc + s^2 wc + cswc + csws + s^2 ws - 2cw - cws^2 + cwcs - swcs\\
& \quad + sws^2 - ws - sw + R\\
& = -3s + 2wc + 2cwc - 2wc + \frac{1}{4}x^4 wc + swc + \frac{1}{2}x^3 wc + sws + \frac{1}{2}x^3 ws + 2cws\\
& \quad - 2ws + \frac{1}{4}x^4ws - 2cw
- 2cwc + 2cw - \frac{1}{4}cwx^4
+ cws + \frac{1}{2} cwx^3
- sws\\
& \quad - \frac{1}{2} swx^3
+ 2swc - 2sw + \frac{1}{4} swx^4
- ws - sw + R\\
& = - 3s + \frac{1}{4}x^4 wu + 3 swc + \frac{1}{2} x^3 wu + 3cws - 3ws - \frac{1}{4} u^{-1} wx^4 \\
& \quad  + \frac{1}{2}u^{-1}wx^3 - 3sw + R\\
& = - 3 s - 3ws - 3sw + 3cws + 3 swc + \frac{1}{2} x^3 wu + \frac{1}{2} u^{-1}wx^3 +  \frac{1}{2} uwx^3w\\
& \quad  + \frac{1}{2} wx^3wu^{-1} + \frac{1}{4} x^4wu - \frac{1}{4}u^{-1}wx^4 + \frac{1}{4} uwx^4w - \frac{1}{4} wx^4wu^{-1}
\end{align*}

If $i = -1$, there is a similar computation.
\end{proofclaim}

We now proceed by increasing complexity of the expected factors; let $n$ be the least integer such that $x^n = 0$ in $\End(V)$.

\subsubsection{Case $n = 2$, $i = 1$}

Suppose $i = 1$ in $\End V$ and $n = 2$, so that $c = 1$ and $s = x$.

The equation (\ref{e:+}) rewrites as $- 3 x = 0$, so $x = 0$; $V$ is clearly $G_0$-trivial.

\subsubsection{Case $n = 2$, $i = -1$}

Suppose $n = 2$ and $i = -1$ in $\End (V)$.

The equation (\ref{e:-}) rewrites as $0 = 3wx + 3xw + 3$, whence $xw+wx = -1$. Therefore $xwx = -x$; on the other hand $Z_1^1(V) = C_V(G_0) \leq C_V(w) = 0$ since $i$ inverts $V$.
The requirements of Criterion \ref{l:SL2:criterion2} are met: $V$ is therefore a direct sum of copies of $\Sym_\Q^1 \Nat \SL_2(\Z) = \Q \otimes_\Z \Nat \SL_2(\Z)$.


\subsubsection{Case $n = 3$, $i = -1$}

Suppose $i = -1$ in $\End(V)$ and $n = 3$, so that $c = 1 + \frac{1}{2}x^2$ and $s =x$.

The equation (\ref{e:-}) rewrites as:
\begin{equation*}\label{e:-3}
0 = 3wx + \frac{3}{2}x^2wx + 3 xw + \frac{3}{2}xwx^2 + \frac{3}{2}x^2w - \frac{3}{2} wx^2 + 3 + \frac{3}{2} x^2\tag{$E_{-3}$}
\end{equation*}
Multiply (\ref{e:-3}) on the left by $x^2$ and on the right by $x$: $3 x^2 wx^2 = 0$. Multiply (\ref{e:-3}) on the left and on the right by $x$: $3xwx^2 + 3x^2wx + 3x^2 = 0$. Finally multiply (\ref{e:-3}) on the left by $x^2$: $3x^2wx + 3x^2 = 0$. So there remains $xwx^2 = 0$, and therefore $\im(x^2) \leq Z_1^1(V) = C_V(G_0) \leq C_V(w) = 0$ since $i$ inverts $V$.

Hence $x^2 = 0$ in $\End(V)$. This case is known.

\subsubsection{Case $n = 3$, $i = 1$}\label{s:+3}

Suppose $n = 3$ and $i = 1$ in $\End (V)$.

The equation (\ref{e:+}) rewrites as $0 = - 3 x - 3 wx - 3 xw + 3wx + \frac{3}{2} x^2 wx + 3xw + \frac{3}{2} x wx^2 = -3 x + \frac{3}{2} x^2 wx  + \frac{3}{2} x wx^2$, or:
\begin{equation*}\label{e:+3}
x^2 wx + xwx^2 = 2x\tag{$E_{+3}$}
\end{equation*}

On the other hand:
\begin{equation*}\label{e:+3'}
wxwx^2 = wswx^2 = - cwsx^2 - swc x^2 = - xwx^2\tag{$E_{+3'}$}
\end{equation*}

\begin{localnotation*}
Let $V_\bot = \Quad(V)$.
\end{localnotation*}

\begin{step}
$V_\bot$ is $G_0$-trivial; $V/V_\bot$ is a direct sum of copies of $\Sym_\Q^2 \Nat G_0$.
\end{step}
\begin{proofclaim}
First recall that $V_\bot$ is a $G_0$-submodule; by the case $n = 2, i = 1$ it is $G_0$-trivial: hence $V_\bot = \Quad(V) = C_V(G_0)$.

Multiply (\ref{e:+3}) on the right by $x$, and find in $\End(V)$: $x^2 w x^2 = 2 x^2$. On the other hand by (\ref{e:+3'}): $wxwx^2 = - xwx^2$. 

These formula still hold of the action on the quotient module $\simpleoverline{V} = V/V_\bot$. By the first paragraph now applied in $\simpleoverline{V}$, $Z_1^2(\simpleoverline{V}) \leq \Quad(\simpleoverline{V}) \leq C_{\simpleoverline{V}}(G_0) = 0$.
But $C_{\simpleoverline{V}}(G_0) = 0$: since the congruence subgroup $G_0'$ acts trivially on the preimage of $C_{\simpleoverline{V}}(G_0)$, so does $G_0$.
So $Z_1^2(\simpleoverline{V}) = 0$ and $\simpleoverline{V}$ meets the requirements of Criterion \ref{l:SL2:criterion2}.
\end{proofclaim}

For the current case $n = 3$ we promised to split the composition series.

\begin{localnotation*}
Let $V_\top = \im (x^2) + \im (wx^2) + \im (xwx^2)$.
\end{localnotation*}
\begin{step}
$V_\top$ is a direct sum of copies of $\Sym_\Q^2 \Nat G_0$; $V/V_\top$ is $G_0$-trivial.
\end{step}
\begin{proofclaim}
$G_0$-invariance of $V_\top$ is obvious thanks to  (\ref{e:+3'}). Recall that in $\End(V)$, $x^2wx^2 = 2x^2$ and $wxwx^2 = - xwx^2$; these still hold in $\End(V_\top)$. Moreover one easily sees  that $Z_1(V_\top) = \im x^2$ and $w\cdot Z_2(V_\top) = \im (wx^2) + \im (xwx^2)$ are in direct sum. So $V_\top$ meets the requirements of Criterion \ref{l:SL2:criterion2} and has the desired form. Since $x^2$ annihilates the quotient module $V/V_\top$, the latter is $G_0$-trivial by the case $i = 1$, $n = 2$.
\end{proofclaim}

Finally let $q$ be a term in $x$ and $w$ which evaluates to $0$ on the $G_0$-trivial line and to $1$ on the adjoint representation (take for instance $\pi_1 + \pi_2 + \pi_3$ with the notations of Criterion \ref{l:SL2:criterion1}).
Since $q$ is $1$ on $V/V_\bot$, $\ker q \leq V_\bot \leq \ker q$. Since $q$ is $0$ on $V/V_\top$, $\ker(q-1) \leq V_\top \leq \ker(q-1)$.
Moreover $q\cdot V \leq V_\top$ so $q(q-1) = 0$ in $\End V$. Hence $V = \ker q \oplus \ker (q-1) = V_\bot\oplus V_\top$.

\begin{localremark*}
One could proceed to module identification by using an action of the Lie ring $\sl_2(\Z)$. Let indeed:
\begin{align*}
y &:= -wxw = -wsw = cws + swc\\
& = wx + \frac{1}{2}x^2 wx + xw + \frac{1}{2} xwx^2;\\
h &:= [x, y] = xwx + x^2 w + \frac{1}{2}x^2wx^2 - wx^2 - \frac{1}{2}x^2wx^2 - xwx\\
& = x^2 w - wx^2
\end{align*}
One finds $[h, x] = x^2 wx + xwx^2 = 2x$ by (\ref{e:+3}), and $[h, y] = -hwxw + wxwh = whxw - wxhw = 2wxw = -2y$. We thus retrieve an action of $\sl_2(\Z)$ on $V$; it extends to an action of $\sl_2(\Q)$, and we could conclude with the techniques of \cite{TV-II}.
\end{localremark*}

\subsubsection{Case $n = 4$, $i = 1$}\label{s:+4}

Suppose $i = 1$ in $\End(V)$ and $n = 4$, so that $c = 1 + \frac{1}{2}x^2$ and $s = x + \frac{1}{6}x^3$.

Bear in mind that $\Cub(V)$ is $G_0$-invariant.
\begin{step}
$\Cub(V)$ and $V/\Cub(V)$ are cubic modules.
\end{step}
\begin{proofclaim}
This is obvious for $\Cub(V)$. For the quotient, we first derive a formula in $\End(V)$.
Multiply equation (\ref{e:+}) on the right by $2 x^3$: one finds $0 = x^3wx^3 + uwx^3wx^3 + wx^3 wx^3 = (u+1) wx^3wx^3 + x^3wx^3$, so dividing on the left by $u+1$, an invertible element in $\End(V)$, one gets $wx^3wx^3 + \frac{1}{2}x^3wx^3 = 0$. Now multiply on the left by $(2w-1)$, and find $x^3 wx^3 = 0$ in $\End V$.

It follows that $x^3\cdot V \leq Z_1^3(V) \leq \Cub(V)$: so the quotient module $V/\Cub(V)$ has length at most $3$.
\end{proofclaim}

\begin{localremark*}
The module $V$ itself need not be cubic.
As a matter of fact pushing the computation to its limits yields in $\End V$ the equation $x^3w + wx^3 + x^2wx + xwx^2 = 2x$, an equation we do not use but which certainly controls the extension $\Cub(V)$-by-$V/\Cub(V)$ in a large measure.
\end{localremark*}

\subsubsection{Case $n = 4$, $i = -1$}\label{s:-4}

Suppose $n = 4$ and $i = -1$ in $\End (V)$.

\begin{step}\label{st:-4:Z13}
$Z_1^3(V) \leq Z_1^2(V)$.
\end{step}
\begin{proofclaim}
Let $a_1 \in Z_1^3(V)$. Then equation (\ref{e:-}) applied to $a_1$ simplifies into: $0 = 3xwa_1 + \frac{3}{2}x^2 w a_1 + 3 a_1$, so $a_1 \in Z_1^2(V)$.
\end{proofclaim}

\begin{step}\label{st:-4:formules}
$(6x^3+x^3wx^3)\cdot V = (wx^2wx^3 - 2xwx^3)\cdot V \leq \Quad(V)$.
\end{step}
\begin{proofclaim}
Multiply equation (\ref{e:-}) on the right by $x^3$:
\begin{align*}
0 & = 3swx^3 + 3cwx^3 - 3wx^3 + 3x^3 + \frac{1}{2}x^3wx^3 + \frac{1}{2}uwx^3wx^3 - \frac{1}{2}wx^3wx^3\\
& = 3 (u-1)wx^3 + 3x^3 + \frac{1}{2}x^3wx^3 + \frac{1}{2}(u-1)wx^3wx^3
\end{align*}
so multiplying on the left by $\frac{x^2}{u-1} = x + O(x^2)$, one finds $3x^2wx^3 + \frac{1}{2}x^2wx^3wx^3 = 0$. Hence $\im (6x^3+x^3wx^3) \leq w\cdot Z_2(V)$; inclusion in $Z_1(V)$ is obvious so $\im (6x^3+x^3wx^3) \leq Z_1^2(V) \leq \Quad(V)$.

Moreover:
\begin{align*}
wx^2wx^3 & = w(2c-2)wx^3\\
& = 2cwsx^3 + 2swcx^3 + 2x^3\\
& = 2 xwx^3 + \frac{1}{3}x^3wx^3 + 2x^3
\end{align*}
whence $\im(wx^2wx^3 - 2xwx^3) = \im (6x^3+x^3wx^3) \leq \Quad(V)$.
\end{proofclaim}

\begin{localnotation*}\
\begin{itemize}
\item
Let $V_1 = \Quad(V)$, $\simpleoverline{\pi}$ be the projection map modulo $V_1$, and $\simpleoverline{V} = V/V_1$.
\item
Let $\simpleoverline{V}_2 = \Quad(\simpleoverline{V})$, $V_2 = \simpleoverline{\pi}^{-1}(\simpleoverline{V}_2)$, and $\doubleoverline{V} = V/V_2$.
\end{itemize}
\end{localnotation*}

We know that $V_1$ and $\simpleoverline{V}_2 \simeq V_2/V_1$ are quadratic $\Q[G_0]$-modules. By Claim \ref{st:-4:formules}, one has in $\End\left(\simpleoverline{V}\right)$, and therefore in $\End \left(\doubleoverline{V}\right)$ as well, $x^3wx^3 = - 6 x^3$ and $wx^2wx^3 = 2 xwx^3$. But this is not enough in order to apply Criterion \ref{l:SL2:criterion2}.

\begin{step}
$Z_1^3(\doubleoverline{V}) = 0$.
\end{step}
\begin{proofclaim}
Let $\doubleoverline{\pi}$ be the projection map modulo $V_2$, $\doubleoverline{V}_3 = \Quad(\doubleoverline{V})$, and $V_3 = \doubleoverline{\pi}^{-1}(\doubleoverline{V}_3)$. It is clear that $V_1$, $V_2/V_1$, and $V_3/V_2$ are quadratic modules. So far we have constructed a quadratic-by-quadratic-by-quadratic submodule $V_3 \leq V$. 

By Claim \ref{st:-4:formules}, one has $x^2w(6x^3 + x^3wx^3) = 0$ in $\End(V_3)$ (actually even in $\End(V)$). Since $V_3$ is quadratic-by-quadratic-by-quadratic, one has in $\End(V_3)$: $x^2wx^3wx^3 = 0$. So $x^2wx^3 = 0$, and $x^3wx^3 = 0$. Always by Claim \ref{st:-4:formules}, one has in $\End(V_3/V_1)$: $6x^3 + x^3wx^3 = 0$ (actually even in $\End(V/V_1)$). So $x^3 = 0$ in $\End(V_3/V_1)$. Hence $V_3/V_1$ is actually a cubic module; by the case $n = 3$, $i = -1$, it is therefore quadratic, i.e. $V_3 = V_2$. This proves $\Quad(\doubleoverline{V}) = 0$.

Finally by Claim \ref{st:-4:Z13}, one has $Z_1^3(\doubleoverline{V}) \leq \Quad(\doubleoverline{V}) = 0$.
\end{proofclaim}

We may now apply Criterion \ref{l:SL2:criterion2} to $\doubleoverline{V}$; the composition series $0 \leq V_1 \leq V_2 \leq V$ has the desired properties.

\begin{localremark*}
Should the author be surprised not being able to apply Criterion \ref{l:SL2:criterion2} in $\simpleoverline{V}$? 
\end{localremark*}

\subsubsection{Case $n = 5$, $i = -1$}\label{s:-5}

Suppose $i = -1$ in $\End(V)$ and $n =5$, so that $c = 1 + \frac{1}{2}x^2 + \frac{1}{24}x^4$ and $s =x + \frac{1}{6} x^3$.

Bear in mind that $\Quad(V)$ is $G_0$-invariant. The author is certainly naive, but he is still puzzled by not having been able to prove that for the expected definition, $\Quart(V)$ is. He did not succeed modulo $\Quad(V)$ nor even modulo $\pi^{-1}(\Quad(V/\Quad(V)))$. So here is a slightly revised definition.

\begin{localnotation*}
Let\inmargin{$\Quart'$} $\Quart'(V) = Z_1^4(V) + (Z_2^3(V) \cap \ker(x^4wx^2w)) + (Z_3^2(V) \cap \ker(x^4wx^2)) + Z_4^1(V)$.
\end{localnotation*}

\begin{step}\label{st:-4:q}
If $a_1 \in Z_1^4(V)$, then $6 a_1 + 6 xwa_1 + x^3wa_1 + xwx^3wa_1 = 0$; in particular $6 a_1 + x^3wa_1 \in \Quad(V)$.
\end{step}
\begin{proofclaim}
Apply equation (\ref{e:-}) to such an element $a_1$, and find: $0 = 3swa_1 + 3cwa_1 - 3wa_1 + 3a_1 + \frac{1}{2}x^3wa_1 + \frac{1}{2}(u-1)wx^3wa_1 = 3 a_1 + 3(u-1)wa_1 + \frac{1}{2}x^3wa_1 + \frac{1}{2}(u-1)wx^3wa_1$. Multiply on the left by $\frac{x}{u-1} = 1+O(x)$: one gets $0 = 3 a_1 + 3xwa_1 + \frac{1}{2}x^3wa_1 + \frac{1}{2}xwx^3wa_1$. So with $b_1 = 6 a_1 + x^3wa_1 \in Z_1(V)$, one has $xwb_1 = -b_1 \in Z_1(V)$, and $b_1 \in Z_1^2(V) \leq \Quad(V)$.
\end{proofclaim}

\begin{step}
$\Quart'(V)$ is $G_0$-invariant.
\end{step}
\begin{proofclaim}
If $a_1\in Z_1^4(V)$, then by Claim \ref{st:-4:q}, $6 a_1+x^3wa_1 \in \Quad(V)$.
Equivalently: for $a_4 \in Z_4^1(V)$, one has $6wa_4 - x^3a_4 \in \Quad(V)$.

So let $a_4 \in Z_4^1(V)$. Write $b_1 = wa_4$ and $q = 6 b_1 + x^3wb_1 \in \Quad(V)$.
Then:
\begin{align*}
x^2 wxa_4 &  = - x^2w\left(s-\frac{1}{6}x^3\right)wb_1\\
& = x^2cwcb_1 + x^2swsb_1 + \frac{1}{6}x^2wx^3wb_1\\
& = x^2wb_1 + \frac{1}{6}x^2w(q-6b_1)\\
& = x^2wb_1 - x^2 wb_1 = 0
\end{align*}
This shows $x a_4 \in Z_3^2(V)$. Moreover $x^4wx^3 a_4 = x^4w(6wa_4 - q) = 0$: hence $x a_4 \in \Quart'(V)$.

Now let $a_3 \in Z_3^2(V) \cap \ker(x^4wx^2)$. Then:
\begin{align*}
x^3wxa_3 & = - x^3wswwa_3 = x^3cwcwa_3 + x^3swswa_3 = x^3wwa_3 + x^4wswa_3\\
& = - x^4cwca_3 = - \frac{1}{2}x^4wx^2 a_3 = 0
\end{align*}
Moreover $x^4wx^2wxa_3 = x^4w(2c-2-\frac{1}{12}x^4)wxa_3 = 2 x^4cwsxa_3 = 2 x^4wx^2 a_3 = 0$, so $x a_3 \in \Quart'(V)$.

Finally let $a_2 \in Z_2^3(V) \cap \ker(x^4wx^2w)$. Then:
\[x^4wxa_2 = - x^4wswwa_2 = x^4cwcwa_2 = \frac{1}{2}x^4wx^2wa_2 = 0\]
and this shows $x a_2 \in \Quart'(V)$, which concludes the verification.
\end{proofclaim}

\begin{step}\label{st:-5:formules}
$x^4wx^4wx^4 = 0$ in $\End (V)$.
\end{step}
\begin{proofclaim}
Multiply equation (\ref{e:-}) on the left and on the right by $x^4$, and find $0 =  \frac{1}{2}x^4wx^4w x^4$.
\end{proofclaim}

Let $V_1 = \Quart'(V)$. Since by Claim \ref{st:-5:formules} one has $x^4wx^4wx^4 = 0$ in $\End(V)$, one finds $x^4wx^4\cdot V \leq Z_1^4(V) \leq V_1$. Let $\simpleoverline{V} = V/V_1$ and let $\simpleoverline{\pi}$ be the projection map modulo $V_1$. Then $x^4wx^4 \cdot \simpleoverline{V} = 0$.

Let $\simpleoverline{V}_2 = \Quart'(\simpleoverline{V})$ and $V_2 = \simpleoverline{\pi}^{-1}(\simpleoverline{V}_2)$. Then $x^4\cdot \simpleoverline{V} \leq \simpleoverline{V_2}$ so $\doubleoverline{V} = V/V_2$ has length at most $4$.

By the case $i = -1$, $n = 4$, one can refine the series $0 \leq V_1 \leq V_2 \leq V$ into another one with the desired properties (we do not know whether powers $\Sym^k$ appear in non-decreasing order in the latter series).

\begin{localremarks*}\
\begin{itemize}
\item
There may be a formula similar to the one given in the final remark of case $n = 4$, $i = 1$ (\S\ref{s:+4}), but this exceeds our computational capacity.
\item
The author cannot answer the following: let $V$ have length $4$. What can one say about $V/\Quart'(V)$?
\end{itemize}
\end{localremarks*}

\subsubsection{Case $n = 5$, $i = 1$}

Suppose $n = 5$ and $i = 1$ in $\End (V)$.

\begin{step}\label{st:+5:Z14}
$Z_1^4(V) \leq Z_1^3(V)$.
\end{step}
\begin{proofclaim}
Apply equation (\ref{e:+}) to $a_1 \in Z_1^4(V)$ and find:
\[0 = \frac{1}{2}x^3wa_1 + \frac{1}{2}uwx^3wa_1 + \frac{1}{2}wx^3wa_1\]
or, $x^3 w a_1 + (u+1)wx^3wa_1 = 0$. Dividing on the left by $u+1$, one gets $\frac{1}{2}x^3 wa_1 + wx^3wa_1 = 0$; since $\frac{1}{2}+w$ is left-inversible in $\End(V)$, we find $x^3 wa_1 = 0$ as claimed.
\end{proofclaim}


\begin{step}\label{st:+5:formula}
$(24x^3wx^4 - x^3wx^4wx^4) \cdot V \leq \Cub(V)$.
\end{step}
\begin{proofclaim}
Multiply equation (\ref{e:+}) on the right by $x^4$:
\begin{align*}
0 & = \frac{1}{2} x^3 w x^4 + \frac{1}{2}uwx^3w x^4 + \frac{1}{2}wx^3w x^4 + \frac{1}{4} x^4w x^4 + \frac{1}{4}uwx^4w x^4 \\
& \quad - \frac{1}{4}wx^4w x^4\\
& = \frac{1}{2} x^3 w x^4 + \frac{1}{2}(u+1)wx^3w x^4 + \frac{1}{4}x^4w x^4 + \frac{1}{4}(u-1)wx^4w x^4
\end{align*}
Multiply on the left by $\frac{4}{u+1} = 2 - x + \frac{1}{12}x^3$:
\begin{align*}\label{e:+5}
0 & = x^3w x^4 - \frac{1}{2} x^4w x^4 + 2 wx^3w x^4 + \frac{1}{2} x^4w x^4 + \frac{u-1}{u+1} wx^4w x^4\\
& = x^3w x^4 + 2 wx^3w x^4 + \tanh\left(\frac{x}{2}\right) wx^4w x^4\\
& = x^3w x^4 + 2 wx^3w x^4 + \frac{1}{2}xwx^4w x^4 - \frac{1}{24}x^3wx^4w x^4\tag{$E_{+5}$}
\end{align*}
Multiply the latter (which we shall use again later) on the left by $(1+w)$:
\begin{align*}
0 & = 3(1+w) x^3w x^4 + \frac{1}{2}(1+w)xwx^4w x^4 - \frac{1}{24}(1+w)x^3wx^4w x^4
\end{align*}
Incidently:
\begin{align*}
(1+w)xwx^4w x^4 & = \left(xwx^4w + w\left(s-\frac{1}{6}x^3\right)wx^4w\right)  x^4\\
& = \left(xwx^4w- swcx^4w- cwsx^4w-\frac{1}{6}wx^3wx^4w\right) x^4\\
& = - \frac{1}{6} (1+w)x^3wx^4w x^4
\end{align*}
So our computation simplifies into:
\[0 =  3(1+w) x^3w x^4 - \frac{1}{8}(1+w)x^3wx^4w x^4\]
Hence $\im (24 x^3 wx^4 - x^3wx^4wx^4) \leq [V, w] \cap \ker x^2 \leq Z_2^2(V) \leq \Cub(V)$.
\end{proofclaim}

\begin{localnotation*}\
\begin{itemize}
\item
Let $V_1 = \Cub(V)$, $\simpleoverline{\pi}$ be the projection map modulo $V_1$, and $\simpleoverline{V} = V/V_1$.
\item
Let $\simpleoverline{V}_2 = \Cub(\simpleoverline{V})$, $V_2 = \simpleoverline{\pi}^{-1}(\simpleoverline{V}_2)$, $\doubleoverline{\pi}$ be the projection map modulo $V_2$, and $\doubleoverline{V} = V/V_2 \simeq  \simpleoverline{V}/\simpleoverline{V}_2$.
\end{itemize}
\end{localnotation*}


\begin{step}\label{st:+5:formula2}
In $\End \doubleoverline{V}$, one has $x^4wx^4 = 24x^4$, $wx^2wx^4 = x^2wx^4$ and $wx^3w x^4 = -6xwx^4$.
\end{step}
\begin{proofclaim}
By Claim \ref{st:+5:formula}, $(24 x^3 wx^4 - x^3wx^4wx^4) \cdot V \leq V_1$, so that $(24 x^3 wx^4 - x^3wx^4wx^4)\cdot \simpleoverline{V} = 0$. 
It follows that $(24x^4 - x^4wx^4) \cdot \simpleoverline{V} \leq Z_1^3(\simpleoverline{V}) \leq \Cub(\simpleoverline{V}) = \simpleoverline{V}_2$. In particular, in $\End(\doubleoverline{V})$, one finds $24x^4 = x^4wx^4$.

Still in $\End(\doubleoverline{V})$, this implies:
\begin{align*}
wx^2wx^4 & = w\left(2c-2-\frac{1}{12}x^4\right)wx^4\\
& = 2 cwx^4 - 2 x^4 - 2 wx^4\\
& = 2 wx^4 + x^2wx^4 + \frac{1}{12}x^4wx^4 - 2 x^4 - 2 wx^4\\
& = x^2wx^4
\end{align*}
And finally, always in $\End(\doubleoverline{V})$ and using equation (\ref{e:+5}) proved in Claim \ref{st:+5:formula}:
\begin{align*}
0 & = x^3w x^4 + 2 wx^3w x^4 + \frac{1}{2}xwx^4w x^4 - \frac{1}{24}x^3wx^4w x^4\\
& = x^3w x^4 + 2 wx^3w x^4 + 12 xwx^4 - x^3wx^4
\end{align*}
so that $wx^3wx^4 = - 6 x wx^4$. All is proved.
\end{proofclaim}

Like in \S\ref{s:-4} this is not quite enough to conclude, as one must control $Z_1^4(\doubleoverline{V})$.

\begin{localnotation*}
Let $\doubleoverline{V}_3 = \Cub(\doubleoverline{V})$, $V_3 = \doubleoverline{\pi}^{-1}(\doubleoverline{V}_3)$, $\tripleoverline{\pi}$ be the projection map modulo $V_3$, and $\tripleoverline{V} = V/V_3 \simeq  \doubleoverline{V}/\doubleoverline{V}_3$; also define $\tripleoverline{V}_4 = \Cub(\tripleoverline{V})$, $V_4 = \tripleoverline{\pi}^{-1}(\tripleoverline{V}_4)$, $\quadrupleoverline{\pi}$ be the projection map modulo $V_4$, and $\quadrupleoverline{V} = V/V_4 \simeq  \tripleoverline{V}/\tripleoverline{V}_4$.
\end{localnotation*}

\begin{step}
$Z_1^4(\quadrupleoverline{V}) = 0$.
\end{step}
\begin{proofclaim}
Let $\quadrupleoverline{V}_5 = \Cub(\quadrupleoverline{V})$ and $V_5 = \quadrupleoverline{\pi}^{-1}(\quadrupleoverline{V}_5)$.

Now $\doubleoverline{V}_5 = V_5/V_2$ is a cubic-by-cubic-by-cubic module. But $\doubleoverline{V_5} \leq V/V_2 = \doubleoverline{V}$ also satisfies $x^4wx^4 = 24x^4$ by Claim \ref{st:+5:formula2}. Hence in $\End \doubleoverline{V}_5$, one has $0 = x^4wx^4wx^4 = 24^2 x^4$ and $\doubleoverline{V}_5$ is actually a quartic module. By the case $n = 4, i=1$ we know that it is actually cubic-by-cubic. Hence $\doubleoverline{V}_5 \leq \doubleoverline{V}_4$ and $V_5 = V_4$.

As a consequence, by Claim \ref{st:+5:Z14}, $Z_1^4(\quadrupleoverline{V}) \leq Z_1^3(\quadrupleoverline{V}) \leq \Cub(\quadrupleoverline{V}) = 0$.
\end{proofclaim}

One may therefore apply Criterion \ref{l:SL2:criterion2} to the action of $G_0$ on $\quadrupleoverline{V} \simeq V/V_4$. Finally the series $0 \leq V_1 \leq V_2 \leq V_3 \leq V_4 \leq $ enjoys the following properties:
\begin{itemize}
\item
$V_1$, $V_2/V_1$, $V_3/V_2$, and $V_4/V_3$ are cubic hence known;
\item
$V/V_4$ is isomorphic to a direct sum of copies of $\Sym^4 \Nat$.
\end{itemize}

We are done. \emph{End of the proof of Theorem \ref{t:SL2:series}.}
\end{proof}

\subsection{A Geometric Interpretation}\label{s:SL2:geometry}


The arguments in \S\ref{s:SL2:geometry} are all due to M. Wolff (in personal communication).


\subsubsection{Short length}\label{s:SL2:trees}

In order to prove Theorem \ref{t:SL2:series} we followed the most naive path: we built consecutive subquotients of $V$ in which we could determine the collection of words in $x$ and $w$. So the proof can provide explicit (additive) generators of the subalgebra $\<\SL_2(\Z)\> \leq \End V$.
Forgetting about $V$, this amounts in a sense to trying to bound the number of additive generators of the quotient $\Z[\SL_2(\Z)]/((u-1)^n)$ of the group ring by the ideal generated by $(u-1)^n$.
Theorem \ref{t:SL2:series} (or more precisely its proof since the statement was over $\Q$) has the following immediate consequence.

\begin{proposition}\label{p:ZSL2/Itf}
For $n \leq 5$, $\Z[\SL_2(\Z)]/((u-1)^n)$ is a finitely generated $\Z$-module.
\end{proposition}

Whether there is a converse proof, from Proposition \ref{p:ZSL2/Itf} to Theorem \ref{t:SL2:series}, is unclear. 
We now give an independent and purely geometric proof of Proposition \ref{p:ZSL2/Itf}.

\begin{proof}[The proof of Proposition \ref{p:ZSL2/Itf} starts here]
The proof makes use of the Bass-Serre tree of $\PSL_2(\Z) = \<w, (uw)\> \simeq \Z/2\Z \ast \Z/3\Z$. 
Since the arity $2$ vertices (associated to $\Z/2\Z$) bear no combinatorial information, we shall forget them and keep only the arity $3$ vertices (associated to $\Z/3\Z$).
In what follows, ``vertex'' will always mean: ternary vertex, and ``edge'' will mean: oriented edge between ternary vertices.

\begin{localnotation*}
Let $V$ be the set of vertices and $E$ be the set of edges.
\end{localnotation*}

$\PSL_2(\Z)$ acts on $V$ with good properties \cite[I, \S4.1, Theorem 7]{SArbresEN}; however the associated action of $\SL_2(\Z)$ is not faithful, so we shall decorate the tree.
%
%
The following must be obvious to the experts.

\begin{localobservation*}
There is a regular action of $\SL_2(\Z)$ on $E' = E \times \{0, 1\}$ lifting the action of $\PSL_2(\Z)$ on $E$.
\end{localobservation*}

We call the elements of $E'$ \emph{coloured edges}.

\begin{localnotation*}\
\begin{itemize}
\item
Let $M = \Z[E']$ be the $\Z$-module freely generated by the elements of $E'$;
\item
let $N \leq M$ be the submodule generated by the elements $(u_1-1)^n\cdot \varepsilon$, for $u_1 \in \{gu^{\pm 1} g^{-1}: g\in G_0\}$ and $\varepsilon\in E'$;
\item
let $Q = M/N$.
\end{itemize}
\end{localnotation*}

By construction the following holds.

\begin{localobservation*}
$\Z[\SL_2(\Z)]/((u-1)^n)$ is finitely generated as a $\Z$-module iff $Q$ is.
\end{localobservation*}
%

Fix some vertex $v_0$. Call \emph{height} of a coloured edge the distance (in the ternary tree $V$) between its origin and $v_0$. We shall prove that coloured edges of bounded height suffice to generate $Q$, by rewriting modulo $N$ every coloured edge of sufficient height as a $\Z$-linear combination of edges of lesser height.
(If $m$ origin-vertices suffice to do it, the number of generators of $Q$ will be bounded above by $6m$.)

Now notice that for any $\varepsilon \in E'$ and $u_1 \in \{g u^{\pm 1}g^{-1}: g \in G_0\}$:
\[\varepsilon = \sum_{k = 1}^{n} (-1)^{k+1} \begin{pmatrix} n\\ k\end{pmatrix} u_1^k \cdot \varepsilon \mod N\]
So in order to show that $\Z[\SL_2(\Z)]/((u-1)^n)$ is finitely generated as a $\Z$-module, it suffices to show that for $\varepsilon$ of sufficient height, there is $u_1 \in \{g u^{\pm 1}g^{-1}: g \in G_0\}$ taking all iterates $u_1\cdot \varepsilon, \dots, u_1^n\cdot \varepsilon$ to (edges congruent with) edges of lesser height.

We then entirely forget about coloured edges and focus on vertices: it suffices to show that isometries of the tree of the form $u_1 \in \{g u^{\pm 1}g^{-1}: g \in G_0\}$ can recursively take far away vertices and their first iterates closer to $v_0$.
The following is obvious when one realises $V$ in the Poincar\'e upper half-plane \cite[I, \S4.2]{SArbresEN}.

\begin{localobservation*}
For any ordered triple $(a, b, c) \in V^3$ of adjacent vertices with $a\neq c$, there is $u_1 \in \{gu^{\pm 1}g^{-1}: g \in \SL_2(\Z)\}$ mapping $a$ to $b$ and $b$ to $c$.
\end{localobservation*}
%

Geometrically, such an element $u_1$ acts as a translation of length $1$ along a geodesic line always turning in the same direction; we call such a transformation a \emph{good map}.

Let $v_0$ be the vertex we fixed and $a_0$ be another vertex at sufficient distance; we are looking for a good map $f$ such that for $i = 1, \dots, n$, $a_i := f^i(a_0)$ is closer to $v_0$ (implicit: than $a_0$ was).

Let $[v_0, a_0] = (v_0, v_1, \dots, v_d = a_0)$ be the minimal path from $v_0$ to $a_0$; we may suppose $d \geq 6$. Fixing arbitrarily one oriented edge ending at $v_0$ but not starting at $v_1$ we may represent the path as its \emph{turn sequence}, i.e. the sequence of lefts and rights $(v_0; t_1, \dots, t_d)$ with $(t_i) \in \{\ell, r\}^d$.

\begin{localobservation*}
If the turn sequence has $k$ consecutive $r$'s or $\ell$'s not starting at $t_1$, then there is a good map $f$ such that $a_i = f^i(a_0)$ is closer to $v_0$ for $i = 1 \dots 2k+1$.
\end{localobservation*}
\begin{proofclaim}
Locate the repetition in the turn sequence; let $f$ be the good map taking the $k\th$ vertex labelled $r$ to the $(k-1)\th$ and the $(k-1)\th$ to the $(k-2)\th$ (this does make sense even if $k = 1$).
\end{proofclaim}

\begin{localconsequence}
Proposition \ref{p:ZSL2/Itf} holds of $n \leq 3$.
\end{localconsequence}
\begin{proofclaim}
There is a good map $f$ taking $a_1, a_2, a_3$ closer to $v_0$.
\end{proofclaim}

\begin{localconsequence}
Proposition \ref{p:ZSL2/Itf} holds of $n \leq 4$.
\end{localconsequence}
\begin{proofclaim}
If the turn sequence has a genuine repetition, i.e. $k$ consecutive similar turns not starting at $t_1$ with $k \geq 2$, then we are done.

So suppose not: up to dyslaterality, the path $[v_0, a_0]$ is $(v_0, \dots, v_{d-4}; \ell, r, \ell, r)$. Let $f$ be the good map taking $a_0 = v_d$ to $v_{d-1}$ and $v_{d-1}$ to $v_{d-2}$. Then $a_1$ and $a_2$ are strictly closer to $v_0$; so is $a_3$ since $d(v_0, a_3) = d(v_0, a_2) + 1 = d(v_0, a_0) - 1$. On the other hand $d(v_0, a_4) = d(v_0, a_0)$, so it suffices to prove that iterates of $b_0 = a_4$ can be taken closer to $v_0$.
But now the path $[v_0, a_4]$ is $(v_0, \dots, v_{d-3}; r, r, \ell)$ with a genuine repetition: whence the claim.
\end{proofclaim}

\begin{localconsequence}
Proposition \ref{p:ZSL2/Itf} holds of $n \leq 5$.
\end{localconsequence}
\begin{proofclaim}
Here again we may assume that there is no genuine repetition in the turn sequence: $[v_0, a_0] = (v_0, \dots, v_{d-6}; \ell, r, \ell, r, \ell, r)$. As above let $f$ be the good map taking $a_0 = v_d$ to $v_{d-1}$ and $v_{d-1}$ to $v_{d-2}$; as above $a_1, a_2, a_3$ are closer to $v_0$, and $a_4$ at constant distance but with a repetition.

So it suffices to show that $c_0 = a_5$ and its iterates can be taken within distance $< d$ of $v_0$; now $[v_0, c_0] = (v_0, \dots, v_{d-2}; r, \ell, \ell)$; be careful that $d(v_0, c_0) = d+1$. Let $g$ be the good map taking $v_{d-2}$ to $v_{d-3}$ and $v_{d-3}$ to $v_{d-4}$. Then letting $c_i = g^i(c_0)$ it is easily checked that:
\begin{itemize}
\item
$[v_0, c_1] = (v_0, \dots, v_{d-2}; \ell, \ell)$;
\item
$[v_0, c_2] = (v_0, \dots, v_{d-3}; \ell, \ell)$;
\item
$[v_0, c_3] = (v_0, \dots, v_{d-6})$;
\item
$[v_0, c_4] = (v_0, \dots, v_{d-4}; r, r, \ell)$;
\item
$[v_0, c_5] = (v_0, \dots, v_{d-5}; r, r, \ell, r, \ell)$.
\end{itemize}
So for $i = 1\dots 5$, $[v_0, c_i]$ is strictly shorter than $d$ or has length $d$ and bears a genuine repetition: we are done.
\end{proofclaim}

\emph{End of the proof of Proposition \ref{p:ZSL2/Itf}.}
\end{proof}

This lovely argument does not yield a composition series; as a matter of fact it does not even provide a way to identify simple $\Q[\SL_2(\Z)]$-modules of short length.
%

\subsubsection{Longer Length}\label{s:ngeq7}

\begin{proposition}\label{p:ngeq7}
If $n \geq 7$, then $\Z[\SL_2(\Z)]/((u-1)^n)$ is \emph{not} finitely generated.
\end{proposition}
\begin{proof}
The ring under consideration admits as a quotient $\Z[\SL_2(\Z)]/((u-1)^7)$, which in turn maps onto:
\[\Z[\SL_2(\Z)]/\left(7, i-1, u^7-1, (u-1)^7\right) \simeq \F_7[\PSL_2(\Z)]/(u^7-1) \simeq \F_7[H]\]
where $H$ is the quotient of $\PSL_2(\Z)$ by the normal closure of $u^7$. Hence $H = \<u, w|(uw)^3 = w^2 = u^7 = 1\>$ is the (``ordinary'') triangle group $(2,3,7)$, which is infinite \cite[\S III.7]{LyndonSchupp}. It follows that $\F_7[H]$ is not finitely generated as a $\Z$-module, and neither is $\Z[\SL_2(\Z)]/((u-1)^7)$.
\end{proof}

It is now clear that the path to Theorem \ref{t:SL2:series} we took is simply hopeless in length $n \geq 7$. Our curiosity is sufficiently aroused to ask the following.

\begin{question*}
What happens when $n = 6$?
\end{question*}

But we prefer to leave the scene before the geometers arrive.



\subsection{Before We Move On}\label{s:fenetre}

The original goal of our work was to study some $\SL_2(\K)$-modules of length $n$. As Proposition \ref{p:ngeq7} shows, the behaviour of $\SL_2(\Z)$-modules of length $n$ grows wild with $n$ and a naive interpretation of our ``two-step methodology'' (see the introduction) over the integers cannot succeed.

Of course working over the ring of integers was too ambitious; over $\F_p$ one may hope to prove Theorem \ref{t:SL2:series} with no restrictions on $n$ (but for decent values of $p$) by arguments from finite group theory.

\begin{question*}
Let $G_1 = \SL_2(\F_p)$ and $V$ be an $\F_p [G_1]$-module of length $n < p$. Does $V$ have a composition series with every factor of the form $\oplus_{I_k}\Sym^k \Nat G_1$?
\end{question*}

The answer must be known \cite{AJLProjective}; 
apparently not so in characteristic $0$.

\begin{question*}
If $V$ is a $\Q [\SL_2(\Q)]$-module of finite length, what happens?
\end{question*}

\section{Scalar Flesh}\label{S:SL2:flesh}

The current section deals with $\SL_2(\K)$-modules. After a few liminary remarks we shall prove Theorem \ref{t:Symn-1NatSL2K} in \S\ref{s:isotypicity}.

\begin{notation*}
Let $\K$ be a field and $G = \SL_2(\K)$; $u, w \in G$ are defined like in \S\ref{s:SL2:criterion}.
Let $U = C_G(u)$, a maximal unipotent subgroup.
\end{notation*}

\begin{notation*}
For $V$ a $G$-module let $Z_0(V) = {0}$ and $Z_{k+1}(V)/Z_k(V) = C_{V/Z_k(V)}(U)$.
\end{notation*}

The length of $V$ is the least $k$ (if any) with $Z_k(V) = V$.

\subsection{A Bitter Remark}\label{s:SL2:Bruhat}

\begin{observation*}
Let $V$ be an $\SL_2(\K)$-module of length $n$. Then $\<G\cdot (Z_1(V) \cap w\cdot Z_{n-1}(V))\>$ has length at most $n-1$.
\end{observation*}
\begin{proofclaim}
We claim that $\<G\cdot (Z_1(V) \cap w\cdot Z_{n-1}(V))\> \leq Z_{n-1}(V)$. Let $a_1 \in Z_1(V) \cap w\cdot Z_{n-1}(V)$ and $g \in G$. Write the Bruhat decomposition $G = B \sqcup BwU$ of $G = \SL_2(\K)$, where $B = N_G(U)$.
Notice that the subgroups $Z_k(V)$ are $B$-invariant, and distinguish two cases:
\begin{itemize}
\item
if $g \in B$, then $g\cdot a_1 \in Z_1(V) \leq Z_{n-1}(V)$;
\item
if $g = bwu$ with obvious notations, then $g\cdot a_1 = bw\cdot a_1 \in Z_{n-1}(V)$,
\end{itemize}
which proves the observation.
\end{proofclaim}

\begin{remark*}
Such an argument for $\SL_2(\Z)$-modules would have delighted us. Yet $\SL_2(\Z)$ has no Bruhat decomposition. Actually our tedious proof of Theorem \ref{t:SL2:series} suggests precisely that in short nilpotence length one may at some cost find something like a weak form of such a decomposition.
\end{remark*}

The observation is not so useful anyway: nothing guarantees that $\overline{V} = V/\<G\cdot (Z_1(V)\cap w\cdot Z_{n-1}(V))\>$ is well-behaved; i.e., we cannot control $Z_1(\overline{V})\cap w\cdot Z_{n-1}(\overline{V})$. (Iterating has no reason to terminate after finitely many steps.)

\subsection{From the Integers to the Rationals}

Here we start using the full Steinberg relations for $\SL_2(\K)$.

\begin{notation*}
For $\lambda \in \K_+$ (resp. $\K^\times$) let $u_\lambda = \begin{pmatrix} 1 & \lambda \\ & 1\end{pmatrix}$ and $t_\lambda = \begin{pmatrix} \lambda \\ & \lambda^{-1}\end{pmatrix}$.
\end{notation*}

\begin{relations*}
$t_\mu u_\lambda t_{\mu^{-1}} = u_{\lambda \mu^2}$ and 
$w t_\lambda w^{-1} = t_{\lambda^{-1}} = t_{\lambda}^{-1}$.
\end{relations*}

\begin{relations*}[Steinberg relations]
$u_\lambda wu_{\lambda^{-1}} wu_\lambda w = t_\lambda$.
\end{relations*}

\begin{notation*}
Suppose that a $G$-module $V$ has length $n$ and is $n!$-divisible and $n!$-torsion-free. Then for $\lambda \in \K$, let $x_\lambda = \log u_\lambda = \sum_{k \geq 1} (-1)^{k+1}\frac{(u_\lambda -1)^k}{k}$.
\end{notation*}

\begin{observation*}
Let $V$ be a $\Q[\SL_2(\Q)]$-module. Suppose that for some unipotent element $u \in \SL_2(\Q)$, $(u-1)^5 = 0$ in $\End V$. Then $V$ has a composition series each factor of which is a direct sum of copies of $\Sym^k \Nat \SL_2(\Q)$ with $k \in \{0, \dots, 4\}$.
\end{observation*}
\begin{proofclaim}
By assumption $u-1$ is, in $\End V$, nilpotent with order say $n$. Since every element in $\Q$ is an integer multiple of a square, it follows from \cite[Variations n$^\circ$5 and n$^\circ$6]{TV-I} that $V$ has $U$-length at most $n$: every element in $U$ has order at most $n$, and we may take logarithms in $\End V$. Then for any integer $a \neq 0$, $e^{a x_{\frac{1}{a}}} = u_{\frac{1}{a}}^a = u = e^x$, and therefore $x_{\frac{1}{a}} = \frac{1}{a} x$, so for any $\lambda \in \Q^\ast$, $x_\lambda = \lambda x$ in $\End V$.

We now show that every term in the composition series (as an $\SL_2(\Z)$-module) provided by Theorem \ref{t:SL2:series} is $\SL_2(\Q)$-invariant; it suffices to show that each term is $T$-invariant where $T$ is the group of diagonal matrices, since $\SL_2(\Q) = \<\SL_2(\Z), T\>$.

But in \emph{any} $\Q[\SL_2(\Q)]$-module of finite length, $\ker x$ is $T$-invariant. This holds since for any rational $\lambda \neq 0$, $x_\lambda = \lambda x$, so they have the same kernel; in particular $\ker x = C_V(u) = C_V(U)$, which is therefore $T$-invariant, and so is $Z_i^j(V)$. In the $n = 5, i = -1$ case one also had to take some intersections (see the definition of $\Quart'$ in \S\ref{s:-5}). But $t_\lambda x^4 wx^2 t_{\lambda^{-1}} = x_{\lambda^2}^4 w x_{\lambda^{-2}}^2 = \lambda^4 x^4 wx^2$, so $\ker (x^4 wx^2)$ is $T$-invariant as well. This shows that the $\SL_2(\Z)$-submodules $\Quad(V)$, $\Cub(V)$, $\Quart'(V)$ are actually $\SL_2(\Q)$-submodules, and the same holds in any subquotient of $V$.

Hence all terms in our composition series are $\Q[\SL_2(\Q)]$-modules. We may focus on one term and assume $V \simeq \oplus_I \Sym_\Q^k\Nat \SL_2(\Z)$ as $\Q[\SL_2(\Z)]$-modules. As we saw the action of $u$ determines that of $u_\lambda$, which by the Steinberg relations determine that of $t_\lambda$, and all these elements act like on $\Sym^k\Nat\SL_2(\Q)$.
\end{proofclaim}

\begin{remark*}
We do not know whether this may hold in longer length or not (see \S\ref{s:fenetre}).
\end{remark*}

\subsection{The Isotypical Case}\label{s:isotypicity}

\renewcommand{\thestep}{\arabic{step}}

\begin{notation*}\
\begin{itemize}
\item
The double factorial $n!!$ is the two-step factorial $n(n-2)(n-4)\dots$
\item
The notation $\oplus_I M$ stands for a direct sum of copies of $M$.
\end{itemize}
\end{notation*}

Recall that $\K$ is $k$-radically closed if for any $\alpha \in \overline{\K}$, $\alpha^k \in \K$ implies $\alpha \in \K$.

\begin{theorem}\label{t:Symn-1NatSL2K}
Let $n\geq 2$ be an integer and $\K$ be a field of characteristic $0$ or $\geq 2n+1$. Suppose that $\K$ is $2(n-1)!!$-radically closed. Let $G = \SL_2(\K)$ and $V$ be a $G$-module. Let $\K_1$ be the prime subfield and $G_1 = \SL_2(\K_1)$. Suppose that $V$ is a $\K_1$-vector space such that $V \simeq \oplus_I \Sym^{n-1} \Nat G_1$ as $\K_1 [G_1]$-modules.

Then $V$ bears a compatible $\K$-vector space structure for which one has $V \simeq\oplus_J \Sym^{n-1}\Nat G$ as $\K [G]$-modules.
\end{theorem}

(We give some slightly different versions after the proof.)
\setcounter{step}{0}
\begin{proof}
Let again $U = C_G(u) = \{u_\lambda: \lambda \in \K_+\}$; let $Z_0 = \{0\}$ and $Z_{k+1}/Z_k = C_{V/Z_k}(U)$; also let\inmargin{$\check{Z}_k$} $\check{Z}_k = Z_k \cap w \cdot Z_{n + 1 - k}$; the former are $B = N_G(U)$-submodules; the latter are only $T = B\cap wBw^{-1}$-submodules. Of course $w \cdot \check Z_k = \check Z_{n + 1 - k}$.

Let\inmargin{$U_1$} $U_1 = U \cap G_1$. Since $V$ has $U_1$-length $n$ and is $n!$-divisible and $n!$-torsion-free, the definition $x = \log u = \sum_{k \geq 1} (-1)^{k+1}\frac{(u -1)^k}{k}$ makes sense in $\End V$.

\begin{localnotation*}
For $a_1 \in Z_1$, $k = 1 \dots n$, let:\inmargin{$\zeta_k(a_1)$}
\[\zeta_k(a_1) = \frac{1}{(n-k)!} x^{n-k} w \cdot a_1\]
\end{localnotation*}

By definition, $\zeta_n(a_1) = w a_1$. Clearly $\zeta_k(a_1) \in Z_k$, but it is not clear a priori whether it lies in $\check Z_k$.
Finally note that $x \zeta_{k+1}(a_1) = (n-k) \zeta_k(a_1)$.

\begin{step}[analysis over $\K_1$]\label{v:Symn-1NatSL2K:st:K1}
$V$ has $U$-length $n$; $V = \oplus_{k = 1}^n \check{Z}_k$. The $\zeta_k$ maps define additive isomorphisms $Z_1 \simeq \check{Z}_k$, whereas $x$ maps $\check{Z}_{k+1}$ to $\check{Z}_k$. Moreover, for any $a_1 \in Z_1$, and any integer $k = 1 \dots n$:
\begin{itemize}
\item
$x^{k-1} \zeta_k (a_1) = (-1)^{n-1}\frac{(n-1)!}{(n-k)!} a_1$;
\item
$w \zeta_k(a_1) = (-1)^{n-k} \zeta_{n+1-k}(a_1)$.
\end{itemize}
In particular any of these formula imply $\zeta_1(a_1) = (-1)^{n-1} a_1$.
\end{step}
\begin{proofclaim}
We keep writing $U_1$ for $U \cap G_1$.
Define $C_0 = \{0\}$, $C_{k+1}/C_k = C_{V/C_k}(U_1)$, and $\check{C}_k = C_k \cap w \cdot C_{n+1-k}$.

By inspection in $\Sym^{n-1} \Nat G_1$, one sees that $\ell_{U_1}(V) = n$, that $V = \oplus_{k = 1}^n \check{C}_k$, that the maps $\zeta_k$ define additive isomorphisms $C_1 \simeq \check{C}_k$ and $x: \check{C}_{k+1} \to \check{C}_k$ likewise, and also that the announced formula are correct. So it suffices to check $C_k = Z_k$ for any $k = 1 \dots n$.

Always by inspection, $C_k = \{v \in V: \forall \lambda \in \K_1^\times$, $t_\lambda\cdot v = \lambda^{n+1-2k} v\}$ (here we use the assumption that the characteristic, if not zero, is $\geq 2n+1$). But for $\lambda\in \K_1^\times$, $\lambda^{n+1-2k}$ lies in $\K_1$ which is the prime field; since the action of $T$ is compatible with the $\Z$-module structure, it is compatible with the $\K_1$-vector space structure. It follows that $C_k$ is $T$-invariant.

Hence $C_1 \leq T\cdot C_1 \leq T \cdot C_V(U_1) = C_V(\{t u t^{-1}: t \in T\})$. Now every element in $\K$ is a square, so $C_1 \leq C_V(U) = Z_1$ and equality follows. Then use induction.
\end{proofclaim}

It therefore makes sense to let $x_\lambda = \log u_\lambda = \sum_{k \geq 1} (-1)^{k+1}\frac{(u_\lambda -1)^k}{k}$.


\begin{step}[a Timmesfeld equation]
For $k = 1 \dots n$, $\lambda \in \K^\times$, $a_1 \in Z_1$:
\[t_\lambda\cdot \zeta_k(a_1) = \frac{(-1)^{n-1}}{n-1} \zeta_k x_{\lambda^{n+1-2k}}\zeta_2(a_1)\]
\end{step}
\begin{proofclaim}
We first show something completely different: let us prove by descending induction on $k = \lfloor \frac{n+1}{2}\rfloor\dots 1$:
\[x_{\lambda^{n+1-2k}}\zeta_{k+1}(a_1) = (n-k) t_\lambda \zeta_k (a_1)\]
\begin{itemize}
\item
Let $k = \lfloor \frac{n+1}{2}\rfloor$. There are two cases, depending on $n$ modulo $2$.
\begin{itemize}
\item
If $n$ is odd, then $n = 2k - 1$, and one has:
\[w \cdot \zeta_k = (-1)^{n-k} \zeta_{n+1-k} = (-1)^{k-1} \zeta_k\]
Depending on $k$ modulo $2$, $w$ inverts or centralises $\check{Z}_k$; in either case $w$ inverts $T$
, so $T$ centralises $\check{Z}_k$. In particular:
\[x_{\lambda^{n+1-2k}}\zeta_{k+1}(a_1) = x \zeta_{k+1}(a_1) = (n-k) \zeta_k(a_1) = (n-k) t_\lambda \zeta_k(a_1)\]
\item
If $n$ is even, then $n = 2k$. Let $\ell \in \K^\times$ be a square root of $\lambda$ and $b_1 \in Z_1$ be such $t_\ell \zeta_k(a_1) = \zeta_k(b_1)$: this exists since $\zeta_k: Z_1 \simeq \check{Z}_k$ is onto. Then:
\begin{align*}
x w t_\ell \zeta_k(a_1) & = x t_\ell^{-1} w \zeta_k(a_1)\\
& = (-1)^{n-k} x t_\ell^{-1} \zeta_{n+1-k}(a_1)\\
& = (-1)^{n-k} t_\ell^{-1} x_\lambda \zeta_{k+1}(a_1)\\
= x w \zeta_k(b_1) & = (-1)^{n-k} x \zeta_{k+1}(b_1)\\
& = (-1)^{n-k} (n-k) \zeta_k(b_1)\\
& = (-1)^{n-k} (n-k) t_\ell \zeta_k(a_1)
\end{align*}
so multiplying by $t_\ell$: $(n-k) t_\lambda \zeta_k(a_1) = x_\lambda \zeta_{k+1}(a_1)$.
\end{itemize}
\item
Suppose the formula holds of $k \geq 2$ and let us prove it at $k-1$. Start with  $x_{\lambda^{2(n+1-2k)}}\zeta_{k+1}(a_1) = (n-k) t_{\lambda^2} \zeta_k(a_1)$ and apply $x_{\lambda^{(n+1-2k)^2}}$:
\begin{align*}
x_{\lambda^{(n+1-2k)^2}} x_{\lambda^{2(n+1-2k)}} \zeta_{k+1}(a_1) & = x_{\lambda^{(n+1-2k)^2}} (n-k) t_{\lambda^2} \zeta_k(a_1)\\
& = (n-k) t_{\lambda^2} x_{\lambda^{(n+1-2k)^2 - 4}} \zeta_k(a_1)\\
& = (n-k) t_{\lambda^2} x_{\lambda^{(n+3-2k)(n-1-2k)}}\zeta_k(a_1)\\
= x_{\lambda^{2(n+1-2k)}} x_{\lambda^{(n+1-2k)^2}} \zeta_{k+1}(a_1) & = x_{\lambda^{2(n+1-2k)}} (n-k) t_{\lambda^{n+1-2k}} \zeta_k (a_1)\\
& = (n-k) t_{\lambda^{n+1-2k}} x \zeta_k(a_1)\\
& = (n-k) (n+1-k) t_{\lambda^{n+1-2k}} \zeta_{k-1}(a_1)
\end{align*}
Multiply by $t_\lambda^{-2}$: $x_{\lambda^{(n+3-2k)(n-1-2k)}}\zeta_k(a_1) = (n+1-k) t_{\lambda^{n-1-2k}}\zeta_{k-1}(a_1)$.
Since $\K$ has all its $(n-1-2k)\th$ roots, rewrite as: $x_{\lambda^{n+3-2k}}\zeta_k(a_1) = (n+1-k) t_\lambda\zeta_{k-1}(a_1)$,
which is the desired formula. This concludes induction and proves the auxiliary formula.
\end{itemize}
We now return to the equation we want: let $k \leq \lfloor \frac{n+1}{2}\rfloor$.

We know that $x$ maps $\check{Z}_{k+1}$ to  $\check{Z}_k$, and we claim that so does $x_\lambda$. Let indeed $\ell \in \K^\times$ be a square root of $\lambda$, so that $x_\lambda = t_\ell x t_{\ell}^{-1}$. Now $\check{Z}_{k+1}$ is $T$-invariant, so $x t_{\ell}^{-1}$ maps $\check{Z}_{k+1}$ to $\check{Z}_k$ which is $T$-invariant, and $x_\lambda$ maps $\check{Z}_{k+1}$ to $\check{Z}_k$.

It follows that $x_{\lambda^{n+1-2k}} x^{n-k-1} w a_1 \in \check{Z}_k$.
We now note, by inspection over $\K_1$, that $\zeta_k x^{k-1} \Id_{\check{Z}_k} = (-1)^{n-1}\frac{(n-1)!}{(n-k)!} \Id_{\check{Z}_k}$.
Therefore:
\begin{align*}
& \quad \frac{(-1)^{n-1}}{n-1} \zeta_k x_{\lambda^{n+1-2k}} \zeta_2(a_1)\\
& = \frac{(-1)^{n-1}}{n-1} \zeta_k x_{\lambda^{n+1-2k}} \frac{1}{(n-2)!} x^{n-2} w a_1\\
& = \frac{(-1)^{n-1}}{(n-1)!} \zeta_k x^{k-1} x_{\lambda^{n+1-2k}} x^{n-k-1} w a_1\\
& = \frac{(-1)^{n-1}}{(n-1)!} (-1)^{n-1} \frac{(n-1)!}{(n-k)!} x_{\lambda^{n+1-2k}} (n-k-1)! \zeta_{k+1}(a_1)\\
& = \frac{1}{n-k} x_{\lambda^{n+1-2k}} \zeta_{k+1}(a_1)\\
& = t_\lambda \zeta_k (a_1)
\end{align*}
So the formula holds of $k \leq \lfloor \frac{n+1}{2}\rfloor$. It then holds as well of $n+1-k$, since:
\begin{align*}
t_\lambda \cdot \zeta_{n+1-k}(a_1) & = (-1)^{n-k} t_\lambda w \zeta_k (a_1)\\
& = (-1)^{n-k} w t_{\lambda^{-1}} \zeta_k (a_1)\\
& = (-1)^{n-k} w \frac{(-1)^{n-1}}{n-1} \zeta_k x_{\lambda^{2k-n-1}} \zeta_2(a_1)\\
& = (-1)^{n-k} \frac{(-1)^{n-1}}{n-1} w \zeta_k x_{\lambda^{2k-n-1}} \zeta_2(a_1)\\
& = \frac{(-1)^{n-1}}{n-1} \zeta_{n+1-k} x_{\lambda^{n+1-2(n+1-k)}} \zeta_2(a_1)
\end{align*}
This completes the proof of the Timmesfeld equation.
\end{proofclaim}

\begin{localnotation*}
For $k = 1 \dots n$, $a_1 \in Z_1$, and $\lambda \in \K^\times$, let:
\[\lambda \cdot \zeta_k(a_1) = \frac{(-1)^{n-1}}{n-1} \zeta_k x_\lambda \zeta_2(a_1)\]
\end{localnotation*}

As $\zeta_k: Z_1 \simeq \check{Z}_k$ is a bijection and $V = \oplus_{k = 1}^n \check{Z}_k$, $\lambda \cdot v$ is defined for any $v \in V$.

\begin{step}
This defines a $\K$-vector space structure compatible with the action of $G$.
\end{step}
\begin{proofclaim}
Additivity in $a_1$ is obvious. So is additivity in $\lambda$: since $\zeta_2(a_1) \in Z_2$, one has $x_{\lambda+\mu}\zeta_2(a_1) = x_\lambda \zeta_2(a_1) + x_\mu \zeta_2(a_1)$.

By the Timmesfeld equation, $\lambda^{n+1-2k}\cdot \zeta_k(a_1) = t_\lambda\cdot \zeta_k(a_1)$. Now $\K$ has all its $(n+1-2k)\th$ roots and $T$ is commutative, so multiplicativity in $\lambda$ follows, and linearity of $T$ as well.
Linearity of $w$ is obvious, since:
\begin{align*}
w (\lambda\cdot \zeta_k(a_1)) & = \frac{(-1)^{n-1}}{n-1} w \zeta_k x_\lambda \zeta_2(a_1)\\
& = \frac{(-1)^{n-1}(-1)^{n-k}}{n-1} \zeta_{n+1-k} x_\lambda \zeta_2(a_1)\\
& = (-1)^{n-k} \lambda \cdot \zeta_{n+1-k}(a_1)\\
& = \lambda \cdot ((-1)^{n-k} \zeta_{n+1-k}(a_1))\\
& = \lambda \cdot (w \zeta_k(a_1))
\end{align*}

To prove linearity of $G$ it therefore suffices to prove linearity of $u$, which amounts to proving that all restrictions $x: \check{Z}_{k+1}\to \check{Z}_k$ are linear, which amounts to proving that all the maps $\zeta_k$ are. Now remember that $\zeta_1(a_1) = (-1)^{n-1} a_1$, so that:
\begin{align*}
\zeta_k(\lambda \cdot a_1) & = (-1)^{n-1} \zeta_k(\lambda\cdot \zeta_1(a_1))\\
& = \frac{(-1)^{n-1}(-1)^{n-1}}{n-1} \zeta_k \zeta_1 x_\lambda \zeta_2(a_1)\\
& = \frac{(-1)^{n-1}}{n-1} \zeta_k x_\lambda \zeta_2(a_1)\\
& = \lambda \cdot \zeta_k (a_1)
\end{align*}
as desired.
\end{proofclaim}
$V$ is therefore a $\K [G]$-module, clearly of the desired form.
\end{proof}

\begin{remark*}
Assuming that $\K$ is quadratically closed might be necessary for even $n$ as well: we could not complete the analysis with $n = 4$ and $\K$ only cubically closed.
\end{remark*}

\begin{remark*}
For the computations properly said, it would be enough to work in characteristic $\geq n$. The assumption that the characteristic, if not zero, is $\geq 2n+1$, is used only in Claim \ref{v:Symn-1NatSL2K:st:K1} of the proof, in order to find a $T$-invariant definition of $\check{C}_k$. When the characteristic is too low we found no such definition. 
But supposing $C_1 = Z_1$ suffices to run the argument.

Alternatively, suppose that $\K$ has characteristic $0$ or $\geq n+1$ and is $2(n-1)!!$-radically closed. Let $\mu \in \K$ be an $(n-1)!!\th$ root of unity; let $\K_\mu = \K_1[\mu]$ and $G_\mu = \SL_2(\K_\mu)$. If $V$ is a $\K_\mu [G]$-module such that $V \simeq \oplus_I \Sym^{n-1} \Nat G_\mu$ as $\K_\mu [G_\mu]$-modules, then the conclusion of Theorem \ref{t:Symn-1NatSL2K} holds since one may characterise $C_k$ as $\{v \in V: t_\mu\cdot v = \mu^{n+1-2k} v\}$, which proves $T$-invariance.
\end{remark*}
%
%
%

As an illustration, here is a cubic analogue of Timmesfeld's Quadratic Theorem.

\begin{corollary*}
Let $\K$ be a quadratically closed field of characteristic $\neq 2, 3$, $G = \SL_2(\K)$, and $V$ be a simple $\Z[G]$-module of $U$-length $3$. Suppose that $C_V(u) = C_V(U)$ for any $u \in U\setminus\{1\}$. Then there exists a $\K$-vector space structure on $V$ making it isomorphic to ${\rm Ad} \PSL_2(\K)$.
\end{corollary*}
\begin{proof}
Analyse over $\K_1$ with Theorem \ref{t:SL2:series}; since $C_V(u) = C_V(U)$ and by simplicity, there are only adjoint summands. Then apply Theorem \ref{t:Symn-1NatSL2K}.
\end{proof}

\medskip
\hrule
\medskip

Future variations will explore minuscule modules for the simple algebraic groups.

\bibliographystyle{plain}
\bibliography{../English/Variationen}

\end{document}